\documentclass[leqno,twoside, 11pt]{amsart}
\usepackage{latexsym,esint}
\usepackage{graphicx}
\usepackage{color}
\usepackage{verbatim}

\theoremstyle{plain}

\numberwithin{equation}{section} \numberwithin{figure}{section}

\newcommand{\dps}{\displaystyle}

\newtheorem{theorem}{Theorem}[section]
\newtheorem{lemma}[theorem]{Lemma}

\title[The von K\'arm\'an theory for incompressible prestrained films]{The incompressible von K\'arm\'an theory \\ for thin prestrained plates}
\author{Hui Li}
\address{Changzhou Vocational Institute of Industry Technology, School of Information Engineering,
28 Minxin Middle Rd., Changzhou, Jiangsu, China 213164}
\email{lihui@ciit.edu.cn} 

\begin{document}

\begin{abstract}
    We derive a new version of the von K\'arm\'an energy and the corresponding Euler-Langrange equations, in the context of thin prestrained plates, under the condition of incompressibility relative to the given prestrain. Our derivation uses the theory of $\Gamma$-convergence in the calculus of variations, building on prior techniques in \cite{Conti-Dolzmann_2006} and \cite{Lewicka-Mahadevan-Pakzad_2011}.
\end{abstract}

\maketitle

\section{Introduction}
Wrinkles and ripples of thin films frequently occur in biological, physical, and chemical systems e.g. at edges of flowers, leaves or torn plastic sheets, in processes associated with growth, swelling or shrinkage, plasticity, following different laws depending on their origin, where the non-trivial energy-minimizing shapes emerge without external forces or boundary conditions. The non-Euclidean elasticity of plates and shells has successfully attempted describing such phenomena (see the recent monograph \cite{Lewicka_2023} and references therein). In this paper, we carry out a new analysis in this context, accounting for the incompressible thin prestrained films.

Recall that for a prestrained film, the local heterogeneous incompatibility of strains is represented by a Riemannian metric $g^h$, posed on the referential configuration $S^h$. To relieve the elastic stresses \cite{Efrati-Sharon-Kupferman_2009, Klein-Efrati-Sharon_2007} the film $S^h$ strives to realize $g^h$ and thus settles with a shape nearest to the isometric immersion of $g^h$. In some deformation processes, it is natural to assume that prestrained films are incompressible with respect to the intrinsic metric $g^h$, i.e. (see \cite{Li_2013}) the only admissible deformations $u^h$ of $S^h$ satisfy the additional constraint $\det\nabla u^h = (\det g^h)^{1/2}$ or equivalently: 
$$\det (\nabla u^h (g^h)^{-1/2}) = 1.$$ 
\noindent Given a surface $S$ in $\mathbb{R}^3$, a film with thickness $h\ll 1$ and midplate $S$ is:
\begin{equation}\label{Sh}
    S^h = \Big\{ z= z' + t \vec n(z')\mid z' \in S, -\frac{h}{2} < t < \frac{h}{2} \Big\},
\end{equation}
where $\vec n(z)$ denotes the unit normal to $S$ at $z'$.  When $S = \Omega \subset \mathbb{R}^2$ then the shell $S^h = \Omega^h$ is called a plate. For a smooth Riemann metric $g^h:S^h\to\mathbb{R}^{3\times 3}_{sym, +}$ and for a deformation $u^h \in W^{1, 2}(S^h, \mathbb R^3)$, its elastic energy is:
\begin{equation}\label{IncomEnergy}
I_{In}^h(u^h) = \frac{1}{h} \int_{S^h} W_{In}\big(\nabla u^h (\sqrt{g^h})^{-1}\big)\;\mbox{d}z,
\end{equation}
where the stored energy density $W_{In}: \mathbb R^{3 \times 3} \to[0, \infty]$ is set to be:
\begin{equation}\label{IncomW}
W_{In}(F) = \left\{\begin{array}{ll} W(F), &\mbox{if}~\det F = 1,\\ +\infty, &\mbox{otherwise}. \end{array}\right.
\end{equation}
Here, the effective energy density $W$ satisfies the standard minimal conditions for physical relevance (\ref{w-ass}). 

To study the asymptotics of the energy minimization  in (\ref{IncomEnergy}) as $h \to 0$, we first determine the scaling exponent $\beta$ such that $\inf I_{In}^h \sim h^{\beta}$ and then derive the $\Gamma$-limit $\mathcal I^{In}_{\beta}$ of $h^{-\beta}I^h_{In}$. 
When $g^h = \mathrm{Id}_3$, the non-Euclidean scenario reduces to the classical nonlinear elasticity investigated in \cite{Conti-Dolzmann_2006, Trabelsi_2005Incom, Trabelsi_2006Incom} at $\beta = 0$, in \cite{Conti-Dolzmann_2007Incom, Conti-Dolzmann_2009} at $\beta = 2$ corresponding to the incompressible version of the Kirchhoff theory; in \cite{Lewicka-Li_2015} where the authors established convergence of equilibria for incompressible elastic plates in the von K\'arm\'an regime $\beta = 4$; and in \cite{Li-Chermisi_2013} where the von K\'arm\'an theory for incompressible elastic shells has been derived. For the non-Euclidean scenario, \cite{Li_2013} investigated shells at $\beta = 2$ subjected to the prestrain given by the metric $g^h$ independent of the thickness. In the present paper, we consider plates at $\beta = 4$ and the prestrain given by growth tensors $a^h$ depending on the thickness defined as in (\ref{GrowthTensorFvK}), see also \cite{Lewicka-Mahadevan-Pakzad_2011}. The growth tensor and the Riemannian metric are connected through $g^h = (a^h)^T a^h$, and when $W$ is, in addition, isotropic, the energy (\ref{IncomEnergy}) can be rewritten as (\ref{aIncomEnergy}). In this context, we first derive the prestrained von K\'arm\'an energy and then compute its Euler-Langrange equations, extending results in \cite{Lewicka-Mahadevan-Pakzad_2011} (see also \cite{Conti-Dolzmann_2007Incom, Conti-Dolzmann_2009, Lewicka-Mora-Pakzad_2010} in the classical case).

To contextualize our findings, we briefly review the general theories when $W_{In} = W$.
When $\beta \geq 2$, the resulting $\Gamma$-limits form an infinite hierarchy. In \cite{ Lewicka-Pakzad_2011, Bhattacharya-Lewicka-Schaffner_2016}, the case $\beta \geq 2$ and $g^h=g(x')$ has been studied; in \cite{Lewicka-Raoult-Ricciotti_2017} it has been shown that if $\beta > 2$, then $\inf I^h \leq C h^4$ which further corresponds to the specific condition on the Riemann curvatures $\{R_{12,ab}\}_{a,b = 1, 2, 3} = 0$ on $\Omega$, and, moreover, if $\beta > 4$, then $\inf I^h_g \leq C h^6$, arising when all curvatures satisfy $R(g)=0$ on $\Omega$; in \cite{Lewicka_2020} these results were extended to $g^h = g: {\Omega}^1, \mathbb R^{3\times 3}_{sym, +})$ varying in the normal direction, and proved that the scaling order of $\inf I^h_g$ relative to $h$ can only be even, i.e. $\inf I^h_g\sim h^{2n}$, obtaining all $\Gamma$-limits in such infinite hierarchy $\{\mathcal{I}_{2n}\}_{n \geq 1}$ of prestrained thin plates; in \cite{Lewicka-Lucic_2020}, the authors investigated the intrinsic metric $g^h$ with more pronounced oscillatory nature; in \cite{Lewicka-Mahadevan-Pakzad_2014, Lewicka-Ochoa-Pakzad_2015, Lewicka-Mahadevan-Pakzad_2017}, the case of plates with even more general structure of $g^h$ has been analyzed; and in \cite{Li_2013}, the author obtained the Kirchhoff theory for shells with the metric $g^h$ invariant in the normal direction. 

When $\beta < 2$, no systematic theory are so far available, but some valuable specific studies have been obtained: in \cite{BenBelgacem-Conti-DeSimone-Muller_2000, BenBelgacem-Conti-DeSimon-Muller_2002, Jin-Sternberg_2001}, the authors studied compression-driven blistering; \cite{Gemmer-Sharon-Shearman-Venkataramani_2016, Gemmer-Venkataramani_2013, Gemmer-Venkataramani_2011} considered buckling; \cite{Venkataraman_2004, Conti-Maggi_2008} investigated origami patterns; \cite{Muller-Olbermann_2014, Olbermann_2019, Olbermann_2016} analyzed conical singularities; and \cite{Bella-Kohn_2014, Bella-Kohn_2017} probed coarsening patterns.

In addition to the above studies under the uniform thickness assumption, plates and shells with varying thickness have been studied, both in classical case \cite{Lewicka-Mora-Pakzad_2009, Li_2017}, and in the non-Euclidean (prestrained) case \cite{Li_2024}.

\section{Statements of the main results}\label{Brief_FvK}
We consider a family of $3$d thin plates: 
\begin{equation}\label{VShellFvK}
\Omega^h = \Big\{x=(x', x_3)\mid x'\in \Omega, \; x_3 \in \big(-\frac{h}{2}, \frac{h}{2}\big)\Big\},
\end{equation}
where $\Omega \subset \mathbb R^2$ is an open, bounded, simply connected domain and $h\ll 1$.
Each $\Omega^h$ undergoes an instantaneous growth, described by $a^h: \Omega^h \to \mathbb R^{3 \times 3}$:
\begin{equation}\label{GrowthTensorFvK}
a^h(x', x_3) = \mbox{Id}_3 + h^2 \epsilon_g(x') + hx_3 \kappa_g(x'),
\end{equation}
where $\epsilon_g, \kappa_g$ are two given smooth matrix fields on $\bar\Omega$. 
For each deformation $u^h \in W^{1, 2}(\Omega^h, \mathbb R^3)$, its  associated elastic energy is given by:
\begin{equation}\label{aIncomEnergy}
I_{In}^h(u^h) = \frac{1}{h} \int_{\Omega^h} W_{In}\big(\nabla u^h (a^h)^{-1}\big)\,\mbox{d}x.
\end{equation}
where $W_{In}$ is as in (\ref{IncomW}) and $W: \mathbb R^{3 \times 3} \to \mathbb R_+$ satisfies the frame indifference, normalization, non-degeneracy, and local regularity conditions:
\begin{equation}\label{w-ass}
\left\{
\begin{minipage}{11cm}
\begin{itemize}
\item [(i)] $W(RF) = W(F)$ for all $R \in SO(3)$ and $F \in \mathbb R^{3 \times 3}$,
\item [(ii)] $W(\mbox{Id}_3) = 0$,
\item [(iii)] $W(F) \geq c \, \mbox{dist}^2 (F, SO(3))$ with some $c>0$,
\item [(iv)] $W$ is $\mathcal C^2$ in a $\delta$-neighborhood of $SO(3)$.
\end{itemize}
\end{minipage}
\right.
\end{equation}

Our first main result is the compactness, lower bound and recovery sequence statements in the theorem below. Therein, $C$ denotes any constant that is independent of $h$ (but it may depend on the given $\Omega$, $W$, $\epsilon_g$, $\kappa_g$).

\begin{theorem}\label{IncompressibleTFvK}
\begin{itemize}
\item [(a)]  Let $u^h \in W^{1, 2}(\Omega^h, \mathbb R^3)$ satisfy:
\begin{equation}\label{IncomEnergysclFvK}
I^h_{In}(u^h) \leq C h^4. 
\end{equation} 
Then there exist rotations $\bar R^h \in SO(3)$ and translations $c^h \in \mathbb R^3$ such that the normalizations:
\begin{equation}\label{IncomNUyFvK}
y^h(x', x_3) = (\bar R^h)^T u^h(x', hx_3) - c^h: \Omega^1 \to \mathbb R^3,
\end{equation}
enjoy the following properties:
\begin{itemize}
\item [(i)] $y^h(x', x_3) \to x'$ in $W^{1, 2}(\Omega^1, \mathbb R^3)$.
\item [(ii)] The scaled displacements:
\begin{equation}\label{IncomSDFvK}
V^h(x') = \frac{1}{h} \fint_{-1/2}^{1/2} y^h(x', t) - x'~ \mathrm{d}t
\end{equation}
converge (up to a subsequence) in $W^{1, 2}(\Omega, \mathbb R^3)$ to the vector field $(0, 0, v)^T$, with the only non-zero out-of-plane scalar component having higher regularity $v \in W^{2, 2}(\Omega, \mathbb R)$.
\item [(iii)] The scaled in-plane displacements $h^{-1}V_{tan}^h$ converge (up to a subsequence) weakly in $W^{1, 2}(\Omega, \mathbb R^2)$ to an in-plane displacement field $w \in W^{1, 2}(\Omega, \mathbb R^2)$.
\item [(iv)] The scaled energies $\frac{1}{h^4} I_{In}^h(u^h)$ have a lower bound:
\[
\liminf_{h \to 0}\frac{1}{h^4}I^h_{In}(u^h) \geq \mathcal I_g^{In}(w, v),
\]
where:
\begin{equation}\label{IncomLmEFvK}
\begin{split}
\mathcal I^{In}_g(w, v) = & \frac{1}{2} \int_{\Omega}\mathcal Q_2^{In}\Big(\mathrm{sym}\,\nabla w + \frac{1}{2}\nabla v \otimes \nabla v - (\mathrm{sym}\,\epsilon_g)_{2 \times 2}\Big) \\
& + \frac{1}{24}\int_{\Omega}\mathcal Q_2^{In}\left(\nabla^2 v + (\mathrm{sym}\,\kappa_g)_{2 \times 2}\right),
\end{split}
\end{equation}
and where the quadratic form $\mathcal Q_2^{In}$ acting on $F \in \mathbb R^{2 \times 2}$ is:
\begin{equation}\label{IncomQ2FvK}
\begin{split}
&\mathcal Q_2^{In}(F) = \min \{\mathcal Q_3(\tilde F)\mid \tilde F \in \mathbb R^{3 \times 3},~ \tilde F_{2 \times 2} = F ~\textnormal{and}~\textnormal{Tr}\,\tilde F = 0\},\\
& \mbox{with} ~~\mathcal Q_3(\tilde F) = \nabla^2W(\mathrm{Id}_3)(\tilde F, \tilde F).
\end{split}
\end{equation}
\end{itemize} 
\item [(b)] Conversely, for each $w \in W^{1, 2}(\Omega, \mathbb R^3)$ and $v \in W^{2, 2}(\Omega, \mathbb R)$, there exist deformations $u^h \in W^{1, 2}(\Omega^h, \mathbb R^3)$ such that (i), (ii), (iii) in part (a) hold with $y^h(x', x_3) = u^h(x', hx_3)$ and the lower bound is realized:
\[
\lim_{h \to 0}\frac{1}{h^4}I^h_{In}(u^h)= \mathcal I_g^{In}(w, v).
\]
\end{itemize}
\end{theorem}

\medskip

Our second main result consists of deriving the associated Euler-Lagrange equations of $\mathcal{I}_g^{In}$ for isotropic materials. These are:
\begin{equation}\label{EL}
\begin{split}
& \Delta^2 \tilde \Phi = - \frac{3 \mu}{2}[v, v] - 3\mu \,\mbox{curl}^T\mbox{curl}(\epsilon_g)_{2 \times 2}\\
& \frac{\mu}{3} \,\Delta^2 v = [v, \Phi] - \frac{\mu}{3} \mbox{div}^T\mbox{div}\Big((\mbox{sym} \, \kappa_g)_{2 \times 2} + \frac{1}{2} \mbox{cof}\,(\mbox{sym}\, \kappa_g)_{2 \times 2}\Big),
\end{split}
\end{equation}
where $\mu$ is one of the Lam\'e constants and $\tilde \Phi \in W^{2, 2}(\Omega, \mathbb R)$ is the Airy stress potential. The system (\ref{EL}) coincides with the limit of the incompatible F\"oppl-von K\'arm\'an equations derived in \cite{Lewicka-Mahadevan-Pakzad_2011} as the Poisson ratio $\nu \to 1/2$.

We organize our paper as follows: Sections 3 and 4 focus on the proof of Theorem \ref{IncompressibleTFvK}, while Section 5 derives the equations (\ref{EL}).

\section{Compactness and the lower bound}
Below, we combine the technique developed for the prestrained plates subject to growth tensor $a^h$ in \cite{Lewicka-Mahadevan-Pakzad_2011} and the method developed for the incompressible Kirchhoff model in \cite{Conti-Dolzmann_2009}. We first recall:

\begin{theorem} \label{TFvK-1}
(Theorem 1.2 and 1.3 of \cite{Lewicka-Mahadevan-Pakzad_2011})
Assume that the deformations $u^h \in W^{1, 2}(\Omega^h, \mathbb R^3)$ satisfy:
\[
   I^h(u^h) = \frac{1}{h} \int_{\Omega^h} W\big(\nabla u^h (a^h)^{-1}\big) \;\mathrm{d}x \leq Ch^4,
\]
relative to the growth tensors $a^h$ as in (\ref{GrowthTensorFvK}) and 
the energy density $W$ as in (\ref{w-ass}). Then there exist $\bar R^h \in SO(3)$ and $c^h \in \mathbb{R}^3$ such that for the normalized deformations $y^h(x, x_3)$ defined by (\ref{IncomNUyFvK}), assertions (i), (ii), (iii) in part (a) of Theorem \ref{IncompressibleTFvK} hold and the scaled energies have the lower bound:
\[
   \liminf_{h \to 0}\frac{1}{h^4}I^h(u^h) \geq \mathcal I_g(w, v),
\]
where:
\begin{equation}\label{LmEFvK}
\begin{split}
\mathcal I_g(w, v) = &\frac{1}{2} \int_{\Omega}\mathcal Q_2\left(\mathrm{sym}\, \nabla w + \frac{1}{2}\nabla v \otimes \nabla v - (\mathrm{sym}\, \epsilon_g)_{2 \times 2}\right) \\
& +  \frac{1}{24}\int_{\Omega}\mathcal Q_2\left(\nabla^2 v + (\mathrm{sym} \,\kappa_g)_{2 \times 2}\right),
\end{split}
\end{equation}
and $\mathcal Q_2$ acting on matrices $F \in \mathbb R^{2 \times 2}$ is defined in:
\begin{equation}\label{Q2FvK}
\begin{split}
& \mathcal Q_2(F) = \min \{\mathcal Q_3(\tilde F)\mid \tilde F \in \mathbb R^{3 \times 3},~ \tilde F_{2 \times 2} = F \}, \\ & \mathcal Q_3(\tilde F) = \nabla^2W(\mathrm{Id}_3)(\tilde F, \tilde F).
\end{split}
\end{equation}
\end{theorem}

\medskip

We now recall the technique from \cite{Conti-Dolzmann_2009}, used in \cite{Li-Chermisi_2013}. For all $k > 0$, let:
\begin{equation}\label{WkFvK}
\begin{split}
& W^k(F) = W(F) + \frac{k}{2}(\det F - 1)^2,
\\ & \mbox{and  } ~ I_k^h(u^h) = \frac{1}{h}\int_{\Omega^h}W^k\big(\nabla u^h(a^h)^{-1}\big)\;\mbox{d}x.
\end{split}
\end{equation}
For each $F \in \mathbb{R}^{3 \times 3}$, consider the quadratic forms associated with $W_k$:
\[
\mathcal Q_3^k(\tilde F) = \nabla^2 W^k(\mbox{Id}_3)(\tilde F, \tilde F) = \mathcal Q_3(\tilde F) + k(\mbox{Tr}\; \tilde F)^2,
\]
and for each $F \in \mathbb{R}^{2 \times 2}$ the corresponding reduced forms:
\[
\mathcal Q_2^k(F) =\min\{\mathcal Q_3^k(\tilde F)\mid \tilde F \in \mathbb R^{3 \times 3}, ~\tilde F_{2 \times 2} = F\}.
\]
According to Lemma 2.1 in \cite{Conti-Dolzmann_2009}, there holds for each $F \in \mathbb{R}^{2 \times 2}$, $k > 0$:
\begin{equation}\label{IneqQ2k}
\mathcal Q_2^{In}(F) - \frac{C}{\sqrt{k}}\|F\|^2 \leq \mathcal Q_2^k(F) \leq \mathcal Q_2^{In}(F), 
\end{equation}
where the constant $C$ is independent of $k$.

\bigskip

\noindent {\bf Proof of Theorem \ref{IncompressibleTFvK}, part (a).} 

\noindent For each $k > 0$, 
the approximate density $W^k$ satisfies (\ref{w-ass}) and also $I_k^h(u^h) \leq I^h_{In}(u^h) \leq Ch^4$. Hence, for each fixed $k$, Theorem \ref{TFvK-1} directly implies (i), (ii), (iii) in part (a). Since only the bound on the energy is utilized in the proof, we see that for different $k$ the same sequences of rotations $\bar R^h$ and translations $c^h$ can be used, and the same subsequences $y^h$. Moreover:
\begin{equation*}
\begin{split}
\liminf_{h \to 0}\frac{1}{h^4}I^h_k (u^h) \geq \mathcal I_g^k(w, v) = & \; \frac{1}{2} \int_{\Omega}\mathcal Q_2^k\Big(\mathrm{sym}\, \nabla w + \frac{1}{2}\nabla v \otimes \nabla v - (\mathrm{sym}\, \epsilon_g)_{2 \times 2}\Big) \\
& + \frac{1}{24}\int_{\Omega}\mathcal Q_2^k\left(\nabla^2 v + (\mathrm{sym} \,\kappa_g)_{2 \times 2}\right).
\end{split}
\end{equation*}
From (\ref{IneqQ2k}), we obtain:
\begin{equation*}
\begin{split}
& \mathcal Q_2^k\Big(\mathrm{sym}\, \nabla w  + \frac{1}{2}\nabla v \otimes \nabla v - (\mathrm{sym}\, \epsilon_g)_{2 \times 2}\Big) \\ 
& \qquad\qquad \geq \mathcal Q_2^{In}\Big(\mathrm{sym}\, \nabla w + \frac{1}{2}\nabla v \otimes \nabla v - (\mathrm{sym}\, \epsilon_g)_{2 \times 2}\Big) 
\\ & \qquad\qquad \quad - \frac{C}{\sqrt{k}}\big\|\mathrm{sym}\, \nabla w + \frac{1}{2}\nabla v \otimes \nabla v - (\mathrm{sym}\, \epsilon_g)_{2 \times 2}\big\|^2,\\
& \mathcal Q_2^k\left(\nabla^2 v + (\mathrm{sym} \,\kappa_g)_{2 \times 2}\right)\\ &\qquad\qquad  \geq \mathcal Q_2^{In}\left(\nabla^2 v + (\mathrm{sym} \,\kappa_g)_{2 \times 2}\right) 
     - \frac{C}{\sqrt{k}}\left\|\nabla^2 v + (\mathrm{sym} \,\kappa_g)_{2 \times 2}\right\|^2.
\end{split}
\end{equation*}
Thus, for each $k>0$there holds:
\begin{equation*}
\begin{split}
& \liminf_{h \to 0}\frac{1}{h^4}I^h_{In} (u^h) \geq \liminf_{h \to 0}\frac{1}{h^4}I^h_k (u^h) \\
& \geq  \frac{1}{2} \int_{\Omega}\mathcal Q_2^{In}\big(\mathrm{sym}\, \nabla w + \frac{1}{2}\nabla v \otimes \nabla v - (\mathrm{sym}\, \epsilon_g)_{2 \times 2}\big) \\
&\qquad \quad - \frac{C}{\sqrt{k}}\big\|\mathrm{sym}\, \nabla w + \frac{1}{2}\nabla v \otimes \nabla v - (\mathrm{sym}\, \epsilon_g)_{2 \times 2}\big\|^2 \mathrm{d} x' \\
& \quad + \frac{1}{24}\int_{\Omega}\mathcal Q_2^{In}\left(\nabla^2 v + (\mathrm{sym} \,\kappa_g)_{2 \times 2}\right) 
     - \frac{C}{\sqrt{k}}\left\|\nabla^2 v + (\mathrm{sym} \,\kappa_g)_{2 \times 2}\right\|^2 \mbox{d}x'.
\end{split}
\end{equation*}
The result follows by passing $k \to \infty$.
\endproof

\section{Construction of the recovery sequence}

We first introduce the following notation. For any $F \in \mathbb R^{2 \times 2}$, by $(F)^{\ast} \in \mathbb R^{3 \times 3}$ we denote the matrix for which $(F)^{\ast}_{2 \times 2} = F$ and $(F)_{i3}^{\ast} = (F)_{3i}^{\ast} = 0$ for $i = 1, 2, 3$. 
The vector $l(F)\in\mathbb{R}^3$ is the unique vector such that:
\[
\mbox{sym}\left(F - (F_{2 \times 2})^{\ast}\right) = \mbox{sym}\big(l(F)\otimes e_3\big).
\]
Since $\mathcal C^{\infty}(\bar \Omega)$ is dense in $W^{2, 2}(\Omega)$ and $\mathcal C^{\infty}(\bar \Omega, \mathbb R^2)$ is dense in $W^{1, 2}(\Omega; \mathbb R^2)$, existence of the desired recovery sequence follows through a diagonal argument from the existence of recovery sequence for smooth $v$ and $w$. The crucial argument is provided by:
\begin{lemma}\label{trace0lemmaFvK}
Let $(v,w) \in \mathcal C^{\infty}(\bar \Omega, \mathbb R\times \mathbb{R}^2)$ and $c_0, c_1 \in \mathcal{C}^{\infty}(\Omega, \mathbb R^3)$ satisfy:
\begin{equation}\label{trace0FvK}
\begin{split}
&\textnormal{Tr}\Big(\big(\mathrm{sym}\, \nabla w + \frac{1}{2} \nabla v \otimes \nabla v - (\mathrm{sym}\, \epsilon_g)_{2 \times 2}\big)^{\ast} \\ & \qquad \qquad\qquad \qquad + \big(\frac{1}{2} |\nabla v|^2 e_3 + c_0\big) \otimes e_3 \Big) = 0,\\
&\textnormal{Tr}\Big(\left(-\nabla^2 v - (\textnormal{sym} \; \kappa_g)_{2 \times 2}\right)^{\ast} + c_1 \otimes e_3\Big) = 0.
\end{split}
\end{equation}
Then, there exist incompressible deformations $u^h \in \mathcal C^1(\Omega^h, \mathbb R^3)$ such that (i), (ii), (iii) in Theorem \ref{IncompressibleTFvK} part (a) hold, with $y^h(x', x_3) = u^h(x', hx_3)$ and:
\begin{equation}\label{limitIhFvK}
\begin{split}
&\lim_{h \to 0}\frac{1}{h^4}I^h_{In}(u^h) =  \frac{1}{2}\int_{\Omega} \mathcal Q_3\Big(\big(\mathrm{sym}\, \nabla w + \frac{1}{2} \nabla v \otimes \nabla v - (\mathrm{sym}\, \epsilon_g)_{2 \times 2}\big)^{\ast} \\ & \qquad\qquad\qquad \qquad\quad + \big(\frac{1}{2} |\nabla v|^2 e_3 + c_0\big) \otimes e_3\Big)\;\mathrm{d}x'\\
& \hspace{2cm} + \frac{1}{24}\int_{\Omega} \mathcal Q_3\left(\left(-\nabla^2 v - (\textnormal{sym} \; \kappa_g)_{2 \times 2}\right)^{\ast} + c_1 \otimes e_3\right)\;\mathrm{d}x'.
\end{split}
\end{equation}
\end{lemma}

\begin{proof}
{\bf 1.} Following \cite{Lewicka-Mahadevan-Pakzad_2011}, we define deformations $u^h_c \in \mathcal C^1(\Omega^h, \mathbb R^3)$:
\begin{equation*}
\begin{split}
u_c^h(x', x_3) = & \begin{bmatrix}x'\\ x_3\end{bmatrix} + \begin{bmatrix}h^2 w(x') \\hv(x')\end{bmatrix} + x_3 \begin{bmatrix}-h\nabla v(x')\\ 0 \end{bmatrix} \\ & + h^2 x_3\left(l(\epsilon_g) + c_0\right) + \frac{1}{2}hx_3^2\left(l(\kappa_g) + c_1\right).
\end{split}
\end{equation*}
Observe that:
\begin{equation}\label{NuchFvK}
\begin{split}
\nabla u_c^h(x', x_3) = & ~\mbox{Id}_3 + h^2 (\nabla w)^{\ast} + h \begin{bmatrix}0 &-\nabla v\\ (\nabla v)^T &0 \end{bmatrix} - h(\nabla^2 v)^{\ast} \\
& + h^2\begin{bmatrix}x_3 \nabla(l(\epsilon_g)+c_0) &l(\epsilon_g) + c_0 \end{bmatrix} \\ & + h^2 \begin{bmatrix}\frac{1}{2}x_3^2 \nabla(l(\kappa_g) + c_1) &x_3(l(\kappa_g )+c_1)\end{bmatrix}.
\end{split}
\end{equation}
We write the desired sequence of scaled deformation on $\Omega^1$ as:
\[
y^h(x', x_3) = y_c^h\big(x', \phi^h(x', x_3)\big),
\]
where function $\phi^h: \Omega^1 \to \mathbb{R}$ is to be determined.
Notice that relations between the original deformations and their scaled versions are given by:
\[
 y^h(x', x_3) = u^h(x', h x_3) \quad \mbox{and} \quad y_c^h(x', \phi^h(x', x_3)) = u_c^h(x', h\phi^h (x', x_3)),
\]
and consequently:
\[
u^h(x', x_3) = u_c^h\big(x', h\phi^h (x', \frac{x_3}{h})\big).
\]
We also have:
\begin{equation*}
\begin{split}
\nabla u^h(x', x_3) = & \; \nabla u_c^h\big(x', h\phi^h(x', \frac{x_3}{h})\big)\times  \\ & \times  \begin{bmatrix}1 &0 &0\\ 0 &1 &0\\ h\partial_1\phi^h\left(x', \frac{x_3}{h}\right) & h\partial_2\phi^h\left(x', \frac{x_3}{h}\right) &\partial_3\phi^h\left(x', \frac{x_3}{h}\right)\end{bmatrix},
\end{split}
\end{equation*}
which leads to:
\begin{equation*}
\begin{split}
& \det \big(\nabla u^h(x', x_3)\big(a^h(x', x_3)\big)^{-1}\big)\\
&= \partial_3 \phi^h\big(x', \frac{x_3}{h}\big)\cdot \det\Big((R^h)^T\nabla u_c^h\big(x', h\phi^h(x', \frac{x_3}{h})\big)\big(a^h\big(x', h\phi^h\big(x', \frac{x_3}{h}\big)\big)\big)^{-1}\Big)\\
&\hspace{2.6cm}\cdot\det\Big(a^h\big(x', h\phi^h\big(x', \frac{x_3}{h}\big)\big)\big(a^h(x', x_3)\big)^{-1}\Big),\\
\end{split}
\end{equation*}
where $R^h$ is an auxiliary $SO(3)$-valued matrix field in:
\[
R^h = e^{hA} = \mbox{Id}_3 + hA + \frac{h^2}{2}A^2 + \mathcal O(h^3) \quad \mbox{ with }~
A = \begin{bmatrix}0 &-\nabla v\\ (\nabla v)^T &0 \end{bmatrix}.
\]
Since the incompressibility of $u^h$ necessitates that $\det (\nabla u^h (a^h)^{-1}) =1$, performing a change of variable and requiring $\phi^h$ to keep the the mid-surface invariant, $\phi^h$ needs to satisfy the following ordinary differential equation:
\begin{equation}\label{ODEFvK}
\left\{
\begin{array}{ll}
\partial_3 \phi^h(x', x_3) = f\left(x', \phi^h(x', x_3), x_3\right)\\
\phi^h(x', 0) = 0,
\end{array}
\right.
\end{equation}
where $f: \Omega \times (-1, 1) \times (-1/2, 1/2) \to \mathbb R$ is given by:
\begin{equation}\label{fxyx3FvK}
\begin{split}
f(x, y, x_3)^{-1} = &\;\det\left((R^h)^T \nabla u_c^h(x', hy)(a^h(x', hy))^{-1}\right)\\ & \cdot \det\Big(a^h(x', hy)\big(a^h(x', hx_3)\big)^{-1}\Big).
\end{split}
\end{equation}

{\bf 2.} In the step below, we will show that for all $h$ is sufficiently small, there exists a unique solution $\phi^h(x', \cdot) \in \mathcal C^1(-1/2, 1/2)$ to (\ref{ODEFvK}) and for each $x' \in \Omega$, such that that the following bounds are satisfied:
\begin{equation}\label{boundsFvK}
\begin{split}
& |\phi^h(x', x_3) - x_3| \leq Ch^3, \quad |\partial_3 \phi^h(x', x_3) - 1| \leq Ch^3, \\ & |\nabla_{x'}\phi^h(x', x_3)|\leq Ch^3. 
\end{split}
\end{equation}
We note that $f$, $\nabla_{x'}f$ and $\nabla_y f$ are Lipschitz in $y$, and that:
\[
f(x', y, x_3)^{-1} = 1 + h^2(y - x_3)\;\mbox{Tr}(\kappa_g) + \mathcal O(h^3)
\]
leading to:
\begin{equation}\label{34'FvK}
f(x', y, x_3) = 1 - h^2(y - x_3)\; \mbox{Tr}(\kappa_g) + \mathcal O(h^3).
\end{equation}
Thus, for all sufficiently small $h$:
\begin{equation}\label{35'FvK}
\frac{1}{2} < f(x', y, x_3) < 2 \quad \mbox{ for all } x' \in \Omega, \; y \in (-1, 1), \; x_3 \in \big(-\frac{1}{2}, \frac{1}{2}\big).
\end{equation}
Further, the derivatives $\nabla_{x'} f$ and $\nabla_y f$ are:
\[
\begin{split}
& \nabla_{x'}f(x', y, x_3)e_{\alpha} = - h^2(y - x_3)\partial_{\alpha}\mbox{Tr}(\kappa_g) + \mathcal O(h^3), \\ &  \nabla_y f(x', y, x_3) = -h^2 \mbox{Tr}(\kappa_g).
\end{split}
\] 
Hence, for $h$ sufficiently small we get:
\begin{equation}\label{38'FvK}
\|\nabla_{x'}f\|_{\infty}+\|\nabla_y f\|_{\infty} \leq 1.
\end{equation}
Consider the Banach space $B = \mathcal C^0(\Omega^1, (-1, 1))$ equipped with $L^{\infty}$ norm, and define the operator $T: B \to B$ as:
\begin{equation}\label{TuFvK}
(Tu)(x', x_3) = \int_0^{x_3}f(x', u(x', t), t)\;\mbox{d}t \quad \mbox{ for all } u \in B,
\end{equation}
which is well posed in virtue of (\ref{35'FvK}) since then $-1 < (Tu)(x', x_3) < 1$. From (\ref{38'FvK}), for each $u, v \in B$ we get:
\[
\begin{split}
&|(Tu - Tv)(x', x_3)| = \left|\int_0^{x_3} f(x', u(x', t), t)\;\mbox{d}t - \int_0^{x_3}f(x', v(x', t), t) \;\mbox{d}t \right|\\
& \leq \int_0^{|x_3|}\big|f(x', u(x', t), t)- f(x', v(x', t), t)\big| \;\mbox{d}t \leq \frac{1}{2}\|u - v\|_{\infty}.
\end{split}
\]
Invoking the Banach fixed point theorem, there exists a unique $\phi^h \in B$ such that $T\phi^h = \phi^h$, which is equivalent to the existence of a solution to (\ref{ODEFvK}).

{\bf 3.} Next we show that $\phi^h$ is $\mathcal C^1$ regular in the tangential direction. For a fixed $\alpha \in \{1, 2\}$, define:
\[
S(u, v)(x', x_3) = \int_0^{x_3} \nabla_{x'}f(x', u(x', t), t)e_{\alpha} + \nabla_y f(x', u(x', t), t)v(x', t)\; \mbox{d}t.
\]
Let $u_0(x', x_3) = 0$ and $v_0(x', x_3) = \partial_{\alpha}u_0(x', x_3) = 0$, and define sequences $\{u_k, v_k\}_{k=1}^\infty$ iteratively as follows:
\[
u_{k + 1} = Tu_{k}, \quad v_{k + 1} = S(u_k, v_k),
\]
so that in particular $\lim_{k \to \infty} u_k = \phi^h$. We will now show that:
\begin{equation}\label{newFvK}
\lim_{k \to \infty} v_k = \partial_{\alpha} \phi^h.
\end{equation}
Note first, that:
\begin{equation}\label{partialTFvK}
\begin{split}
& \partial_{\alpha} (Tu)(x', x_3) \\ & \; = \int_0^{x_3} \nabla_{x'} f(x', u(x', t), t)e_{\alpha}+ \nabla_y f(x', u(x', t), t)\partial_{\alpha} u(x', t) \; \mbox{d}t \\ & 
\; = S(u, \partial_{\alpha} u).
\end{split}
\end{equation}
In fact, an inductive argument implies that for each $k \geq 1$, thee holds $v_k(x', x_3) = \partial_{\alpha}u_k(x', x_3)$, because:
\begin{itemize}
\item[(i)] For $k = 0$, this identity holds by the definition of $v_0$.
\item[(ii)] If  $v_k = \partial_{\alpha} u_k$ then, by
(\ref{partialTFvK}):
$v_{k+1} = S(u_k, v_k) = S(u_k, \partial_{\alpha}
u_k) = \partial_{\alpha} (T u_k) = \partial_{\alpha} u_{k+1}.$
\end{itemize}

\smallskip

\noindent Now, for each $v \in \mathcal C(\Omega^1)$, define:
\begin{equation*}
\begin{split}
(Rv)(x', x_3) & \; = S(\phi^h, v)(x', x_3) \\ & = \int_0^{x_3} \nabla_{x'} f(x', \phi^h(x',
t), t)e_{\alpha} + \nabla_y f(x', \phi^h(x', t), t) v(x', t) ~\mbox{d}t.
\end{split}
\end{equation*}
For any two $v_1, v_2$, we have:
\begin{equation*}
\begin{split}
| (R v_1 - R v_2)(x', x_3)| &\; = \left|\int_0^{x_3} \nabla_y f(x, \phi^h(x', t),
t)(v_1(x', x_3) - v_2(x', x_3))\;\mbox{d}t\right| \\ & \leq \frac{1}{2}\|v_1 -
v_2\|_{\infty}. 
\end{split}
\end{equation*}
Thus the mapping $R$ is a contraction, and hence it has a fixed point $z$. Further:
\[
\begin{split}
v_{k + 1} - Rv_k = & \; S(u_k, v_k) - S(\phi^h, v_k)\\
= & \int_0^{x_3} \nabla_x f(x', u_k(x', t), t)e_{\alpha} + \nabla_y f(x', u_k(x',
t), t) v_k(x', t)\;\mbox{d}t\\
& - \int_0^{x_3} \nabla_{x'} f(x', \phi^h(x', t), t)e_{\alpha} +
\nabla_y f(x', \phi^h(x', t), t)v_k(x', t)\;\mbox{d}x'\\
\leq &\; L_1\|u_k - \phi^h\|_{\infty} + L_2 \|u_k -
\phi^h\|_{\infty}\|v_k\|_{\infty}.
\end{split}
\]
where $L_1, L_2$ are the Lipschitz constants of $\nabla_{x'} f(x',
y, x_3)$ and $\nabla_y f(x', y, x_3)$ with respect to $y$. Since $\|u_k -
\phi^h\|_{\infty} \to 0$ as $k \to \infty$, we may use Ostrowski
theorem on approximate iteration \cite{walter_2013}, and conclude that $v_k
\to z$ uniformly in $\Omega \times (-1/2, 1/2)$.
Therefore, there must hold: $z = \partial_{\alpha} \phi^h(x', x_3)$ and:
\begin{equation*}
\lim_{k \to \infty} v_k = \partial_{\tau} \phi^h,
\end{equation*}
as claimed.

\smallskip

\textbf{4.} We start this step by proving:
\begin{equation}\label{bdphihtanFvK}
|\nabla_{x'} \phi^h (x', x_3)| < C, \qquad \mbox{ for all } x' \in \Omega, \;  x_3 \in \big(-\frac{1}{2}, \frac{1}{2}\big).
\end{equation}
In view of (\ref{newFvK}), the boundedness of the sequence $\{v_k\}$ and (\ref{38'FvK}) we get:
\[
\begin{split}
| v_1(x', x_3)| &= \left|\int_0^{x_3} \nabla_{x'} f(x', u_0(x', t), t) e_\alpha + \nabla_y f(x', u_0(x', t), t) v_0(x', t)\;\mbox{d}t\right| \\ & \leq
\frac{1}{2},\\
    | v_2(x', x_3)| & = \left|\int_0^{x_3} \nabla_{x'} f(x', u(x', t), t) e_\alpha +
\nabla_y f(x', u(x', t), t) v_1(x', t) \; \mbox{d}t\right|\\
& \leq \int_0^{|x_3|} 1 +  \frac{1}{2}\; \mbox{d}s \leq
\frac{1}{2} + \frac{1}{2^2}.
\end{split}
\]
Inductively, it follows that:
$$ \|v_k\|_{\infty} \leq \sum_{i = 1}^{k} \frac{1}{2^i} = 1, $$
implying (\ref{bdphihtanFvK}).
By (\ref{ODEFvK}) and (\ref{34'FvK}), we now see that:
\begin{equation}\label{41'FvK}
\partial_3 \phi^h(x', x_3) - 1 = - h^2(\phi^h - x_3) \mbox{Tr}(\kappa_g) + \mathcal O(h^3),
\end{equation}
which yields:
\begin{equation*}
\begin{split}
& |\partial_3 \phi^h(x', x_3) - 1| \leq C h^2,
\\ &
|\phi^h(x', x_3) - x_3| = \Big|\int_0^{x_3}(\partial_3 \phi^h(x', t)- 1)\; \mbox{d}t\Big| \leq  C h^2.
\end{split}
\end{equation*}
Substituting the second bound above into the first one, we obtain:
\begin{equation}\label{44'FvK}
|\partial_3 \phi^h - 1| \leq Ch^3,
\end{equation}
which is the second inequality in (\ref{boundsFvK}). Using (\ref{44'FvK}) gives the first inequality in (\ref{boundsFvK}). To show the estimate for tangential derivatives, we study:
\begin{equation}\label{46'FvK}
\begin{split}
&\partial_{\alpha}\phi^h(x', x_3) = \int_0^{x_3} \partial_{\alpha} f(x', \phi^h, t)\;\mbox{d}t \\ & = \int_0^{x_3} \partial_{\alpha}\big(1 - h^2(\phi^h - t)\mbox{Tr}(\kappa_g) \big) + \mathcal O(h^3)\; \mbox{d}t\\
& = \int_0^{x_3}-h^2(\partial_{\alpha} \phi^h)\mbox{Tr}(\kappa_g) - h^2(\phi^h - t)\partial_{\alpha}(\mbox{Tr}(\kappa_g)) + \mathcal O(h^3)\; \mbox{d}t.
\end{split}
\end{equation} 
Based on (\ref{bdphihtanFvK}) and the first two inequalities in (\ref{boundsFvK}), we get:
\begin{equation}\label{45'FvK}
|\partial_{\alpha} \phi^h| \leq Ch^2.
\end{equation}
Applying (\ref{45'FvK}) in (\ref{46'FvK}), still with the first two inequalities in (\ref{boundsFvK}), we conclude the last inequality in (\ref{boundsFvK}).

\smallskip

{\bf 5.} To prove that the scaled deformation family $\{y^h\}$ realizes the limiting energy $\mathcal I_g^{In}$, we recall that:
\begin{equation}\label{GraduFvK}
\nabla u^h(x', hx_3) = \nabla u_c^h(x', h\phi^h)\begin{bmatrix}1 &0 &0\\ 0 &1 &0\\ h \partial_1\phi^h & h\partial_2\phi^h & \partial_3 \phi^h\end{bmatrix},
\end{equation}
and obtain:
\begin{equation*}
\begin{split}
& (R^h)^T \nabla u_c^h(x', h\phi^h)(a^h(x', h\phi^h))^{-1} 
\\ & \qquad = \mbox{Id}_3 + h^2 \big((\nabla w)^{\ast} + (l(\epsilon_g) + c_0) \otimes e_3 - \epsilon_g - \frac{1}{2}A^2 \big)\\
& \qquad \quad + h^2 \phi^h \big((l(\kappa_g) + c_1)\otimes e_3 - \kappa_g - (\nabla^2 v)^{\ast}\big) + \mathcal O(h^3), \\ & 
a^h(x', h\phi^h)(a^h(x', hx_3))^{-1} = \mbox{Id}_3 + h^2(\phi^h - x_3)\kappa_g + \mathcal O(h^4).
\end{split}
\end{equation*}
Owing to the bounds in (\ref{boundsFvK}) and the above identities, we infer that:
\begin{equation*}
\begin{split}
&(R^h)^T \nabla u^h(x', hx_3)\big(a^h(x', hx_3)\big)^{-1}\\
&=(R^h)^T \nabla u_c^h(x', h \phi^h)\big(a^h(x', h\phi^h)\big)^{-1}a^h(x', h\phi^h)\big(a^h(x', hx_3)\big)^{-1}\\
&\quad + (R^h)^T \nabla u_c^h(x', h\phi^h)e_3 \otimes \begin{bmatrix}h\partial_1\phi^h &h \partial_2 \phi^h &\partial_3 \phi^h - 1\end{bmatrix}^Ta^h(x', hx_3)^{-1}\\
&= \mbox{Id}_3 + h^2 \Big((\nabla w)^{\ast} - \epsilon_g - \frac{1}{2}A^2 + (l(\epsilon_g)+ c_0)\otimes e_3\Big)\\
& \quad + h^2 x_3\left(-\kappa_g-(\nabla^2 v)^{\ast} + \left(l(\kappa_g)+c_1\right)\otimes e_3 \right) + \mathcal O(h^3).
\end{split}
\end{equation*}
Recall the quadratic form $\mathcal Q_3(\tilde F) = \nabla^2 W (\mathrm{Id_3})(\tilde F, \tilde F) = \mathcal Q_3(\mathrm{sym}\, \tilde F)$. Taylor expanding $W$ around $\mathrm{Id_3}$ and utilizing the uniform boundedness of all the terms involved, we get:
\begin{equation*}
\begin{split}
\lim_{h \to 0}\frac{1}{h^4}& I^h_{In}(u^h) \\ & = ~\frac{1}{2}\int_{\Omega}\mathcal Q_3\Big( \mathrm{sym}\,\big((\nabla w)^{\ast} - \epsilon_g - \frac{1}{2}A^2 + (l(\epsilon_g) + c_0)\otimes e_3\big)\Big)\;\mbox{d}x'\\
& \quad + \frac{1}{24} \int_{\Omega}\mathcal Q_3\big(\mathrm{sym}\left(-\kappa_g - (\nabla^2 v)^{\ast} + (l(\kappa_g) + c_1)\otimes e_3\right)\big)\;\mbox{d}x'.
\end{split}
\end{equation*} 
Notice that: 
\[
A^2 = - (\nabla v \otimes \nabla v)^{\ast} - |\nabla v|^2 e_3 \otimes _3. 
\]
Thus:
\begin{equation*}
\begin{split}
& \mathrm{sym}\,\Big((\nabla w)^{\ast} - \epsilon_g - \frac{1}{2}A^2 + (l(\epsilon_g) + c_0)\otimes e_3\Big) \\ & \quad = \Big(\mathrm{sym}\, \nabla w + \frac{1}{2} \nabla v \otimes \nabla v - (\mathrm{sym}\, \epsilon_g)_{2 \times 2}\Big)^{\ast} 
+ \mathrm{sym}\, \big(\frac{1}{2} |\nabla v|^2 e_3 + c_0\big) \otimes e_3,  \\
& \mathrm{sym}\left(-\kappa_g - (\nabla^2 v)^{\ast} + (l(\kappa_g) + c_1)\otimes e_3\right) \\ & \quad = \left(-(\mathrm{sym}\, \kappa_g)_{2 \times 2} - \nabla^2 v\right)^{\ast} + \mathrm{sym}(c_1 \otimes e_3).
\end{split} 
\end{equation*}
As $\mathcal Q_3$ depends only on the symmetric part of its argument, the above implies the lemma.
\end{proof}

\medskip

Using the same technique as in \cite{Conti-Dolzmann_2009, Li-Chermisi_2013}, we obtain:
\begin{lemma}\label{smoothvwFvK}
For each $v \in C^{\infty}(\bar \Omega, \mathbb R)$ and $w \in C^{\infty}(\bar \Omega; \mathbb R^2)$, there exists a family $\{u^h \in \mathcal C^1(\Omega^h, \mathbb R^3)\}_{h\to 0}$ satisfying Theorem \ref{IncompressibleTFvK} part (b).
\end{lemma}
\begin{proof}
By the positive defniteness of $\mathcal Q_3$, there exist vector fields $c_0, c_1 \in L^2(\Omega, \mathbb{R}^3)$ such that:
\begin{equation} \label{Q2InQ3Trace0}
\begin{split}
&\mathcal Q_2^{In}\Big(\mathrm{sym}\,\nabla w + \frac{1}{2}\nabla v \otimes \nabla v - (\mathrm{sym}\,\epsilon_g)_{2 \times 2}\Big) \\
& \qquad = \mathcal Q_3\Big(\big(\mathrm{sym}\,\nabla w + \frac{1}{2}\nabla v \otimes \nabla v - (\mathrm{sym}\,\epsilon_g)_{2 \times 2}\big)^{\ast} + c_0 \otimes e_3\Big),\\
& ~\qquad \mathrm{with} ~~ \mathrm{Tr}\Big(\big(\mathrm{sym}\,\nabla w + \frac{1}{2}\nabla v \otimes \nabla v - (\mathrm{sym}\,\epsilon_g)_{2 \times 2}\big)^{\ast} + c_0 \otimes e_3\Big) = 0;\\
& \mathcal Q_2^{In}\left(\nabla^2 v + (\mathrm{sym}\, \kappa_g)_{2 \times 2}\right) = \mathcal{Q}_3\left((\nabla^2 v + (\mathrm{sym}\, \kappa_g)_{2 \times 2})^{\ast} + c_1 \otimes e_3\right)\\
& ~\qquad \mathrm{with} ~~ \mathrm{Tr}\left((\nabla^2 v + (\mathrm{sym}\, \kappa_g)_{2 \times 2})^{\ast} + c_1 \otimes e_3\right) = 0.
\end{split}
\end{equation}
Comparing the trace-zero condition in (\ref{trace0FvK}) and (\ref{Q2InQ3Trace0}, we let $\bar c_0, \bar c_1 \in L^2(\Omega, \mathbb R^3)$ be such that:
\[
c_0 = \frac{1}{2} |\nabla v|^2 e_3 + \bar c_0, \quad c_1 = - \bar c_1.
\]
In view of (\ref{trace0FvK}) and (\ref{limitIhFvK}), we rewrite (\ref{Q2InQ3Trace0}) as:
\begin{equation*}
\begin{split}
& \mathcal Q_2^{In}\Big(\mathrm{sym}\,\nabla w + \frac{1}{2}\nabla v \otimes \nabla v - (\mathrm{sym}\,\epsilon_g)_{2 \times 2}\Big) \\
& \quad = \mathcal Q_3\Big(\big(\mathrm{sym}\,\nabla w + \frac{1}{2}\nabla v \otimes \nabla v - (\mathrm{sym}\,\epsilon_g)_{2 \times 2}\big)^{\ast} + \big(\frac{1}{2} |\nabla v|^2 e_3 + \bar c_0\big) \otimes e_3\Big)\\
& \quad \mathrm{with} ~~ \mathrm{Tr}\Big(\big(\mathrm{sym}\,\nabla w + \frac{1}{2}\nabla v \otimes \nabla v - (\mathrm{sym}\,\epsilon_g)_{2 \times 2}\big)^{\ast} \\ 
& \qquad\qquad\qquad \qquad \qquad \qquad  + \big(\frac{1}{2} |\nabla v|^2 e_3 + \bar c_0\big) \otimes e_3\Big) = 0,\\
&\mathcal Q_2^{In}\left(\nabla^2 v + (\mathrm{sym}\, \kappa_g)_{2 \times 2}\right) = \mathcal{Q}_3\left(-(\nabla^2 v + (\mathrm{sym}\, \kappa_g)_{2 \times 2})^{\ast} + \bar c_1 \otimes e_3\right)\\
&\quad \mathrm{with} ~~ \mathrm{Tr}\left(-(\nabla^2 v + (\mathrm{sym}\, \kappa_g)_{2 \times 2})^{\ast} + \bar c_1 \otimes e_3\right) = 0.
\end{split}
\end{equation*}
From the above and recalling the regularity of $v$, $w$, $\epsilon_g$ and $\kappa_g$, we obtain that $\bar c_i \cdot e_3 \in \mathcal C^{\infty}(\bar \Omega)$ for $i = 0, 1$. Let ${c_i^n}_{n=1}^\infty$ be sequences of in $\mathcal C^{\infty}(\bar \Omega)$ converging to $\bar c_i$ in $L^2(\Omega, \mathbb R^3)$, with $c_i^n \cdot e_3 = \bar c_i \cdot e_3$ for $i= 0, 1.$ Thus, $c_0^n, c_1^n$ satisfies the trace zero condition in the above displayed identities, implying that:
\begin{equation*}
\begin{split}
&\lim_{n \to \infty}\int_{\Omega} \mathcal Q_3\Big(\big(\mathrm{sym}\, \nabla w + \frac{1}{2} \nabla v \otimes \nabla v - (\mathrm{sym}\, \epsilon_g)_{2 \times 2}\big)^{\ast} \\ 
& \hspace{6cm} + \big(\frac{1}{2} |\nabla v|^2 e_3 + c_0^n\big) \otimes e_3\Big)~\mbox{d}x' \\
&= \int_{\Omega} \mathcal Q_3\Big(\big(\mathrm{sym}\, \nabla w + \frac{1}{2} (\nabla v )^{\otimes 2} - (\mathrm{sym}\, \epsilon_g)_{2 \times 2}\big)^{\ast} + \big(\frac{1}{2} |\nabla v|^2 e_3 + \bar c_0\big) \otimes e_3\Big)\;\mbox{d}x',\\
& \lim_{n \to \infty} \int_{\Omega} \mathcal Q_3\left(-\left(\nabla^2 v + (\textnormal{sym} \; \kappa_g)_{2 \times 2}\right)^{\ast} + c_1^n \otimes e_3\right)\;\mbox{d}x' \\
& = \int_{\Omega} \mathcal Q_3\Big(-\left(\nabla^2 v + (\textnormal{sym} \; \kappa_g)_{2 \times 2}\right)^{\ast} + \bar c_1 \otimes e_3\Big)\;\mbox{d}x'.
\end{split}
\end{equation*}
The required recovery sequence $u^h$ is now actualized through applying Lemma \ref{trace0lemmaFvK} to $v$, $w$, $c_0^n$, $c_1^n$, and taking a diagonal sequence.
\end{proof}

\smallskip

Finally, based on the density of $C^{\infty}(\bar \Omega, \mathbb R)$ and $C^{\infty}(\bar \Omega, \mathbb R^2)$ in $W^{2,2}(\Omega, \mathbb R)$ and $W^{1, 2}(\Omega, \mathbb R^2)$ respectively, Lemma \ref{smoothvwFvK} implies  Theorem \ref{IncompressibleTFvK} part (b) through a diagonal argument. 

\bigskip

\section{The Euler-Lagrange Equations for prestrained isotropic thin plates}
In this section, we write the Euler-Lagrange equations of the limiting energy $\mathcal I_g^{In}$ for isotropic plates. Recall that the effective energy density $W$ of such obeys the additional property:
\begin{equation*}
\forall F \in \mathbb R^{3 \times 3}, \; R \in SO(3) \quad W(FR) = W(F).
\end{equation*}
The quadratic form $\mathcal Q_3$ thus takes the form (see e.g. \cite{Friesecke-James-Muller_2006}):
\begin{equation}
\mathcal Q_3(F) = 2 \mu |\mathrm{sym} F|^2 + \lambda |\mathrm{Tr} F|^2,
\end{equation}
where $\mu$ and $\lambda$ are the Lam\'e constants, and further:
\begin{equation*}
\mathcal Q_2^{In} (F) = 2 \mu\left(|\mbox{sym} \, F|^2 + |\mbox{Tr}\, F|^2\right).
\end{equation*}
Through the same calculation as in \cite{Lewicka-Mahadevan-Pakzad_2011}, we obtain the Euler-Lagrange equations of $\mathcal I_g^{In}$ in (\ref{IncomLmEFvK}) as follows:
\begin{equation}\label{ELIncomFvK}
\begin{array}{ll}
{\displaystyle \Delta^2 \tilde \Phi = -\frac{3 \mu}{2}[v, v] - 3\mu\, \mbox{curl}^T\mbox{curl}\,(\epsilon_g)_{2 \times 2}}\vspace{1mm}\\
{\displaystyle \frac{\mu}{3} \Delta^2 v = [v, \tilde \Phi] - \frac{\mu}{3} \mbox{div}^T\mbox{div}\Big((\mbox{sym} \; \kappa_g)_{2 \times 2} + \frac{1}{2}\mbox{cof}(\mbox{sym} \; \kappa_g)_{2 \times 2}\Big)}, 
\end{array}
\end{equation}
where we use the formula:
\[
\begin{split}
\mbox{cof} \; \nabla^2 \tilde \Phi = 2 \mu\Big(& \mbox{sym} \nabla w + \frac{1}{2}\nabla v\otimes \nabla v \\ & -(\mbox{sym} \nabla \epsilon_g)_{2 \times 2} \big(\mbox{div}\, w + \mbox{Tr}\big(\frac{1}{2}\nabla v \otimes \nabla v - (\mbox{sym} \;\epsilon_g)_{2 \times 2}\big)\big)\mbox{Id}_3\Big),
\end{split}
\]
to recover the in-plane displacement $w$ from the Airy stress potential $\tilde \Phi \in W^{2, 2}(\Omega, \mathbb R)$. Denote:
\[
\hat \Psi = \nabla^2 v + (\mbox{sym}\; \kappa_g)_{2 \times 2}.
\]
The natural boundary conditions on $\partial\Omega$, associated with (\ref{ELIncomFvK}) are:
\begin{equation*}
\begin{array}{ll}
\tilde \Phi = \partial_{\vec n} \tilde \Phi = 0,\\
\langle \hat \Psi : \vec n \otimes \vec n \rangle + \frac{1}{2} \langle \hat \Psi:\tau \otimes \tau\rangle  = 0, \\
(1- \nu)\partial_{\tau}\langle\hat \Psi: \vec n \otimes \tau\rangle + \mbox{div}\big(\hat \Psi + \frac{1}{2} \mbox{cof} \,\hat \Psi\big)\vec n = 0,
\end{array}
\end{equation*}
where $\vec n \in \mathbb R^2$ denotes the unit normal and $\tau$ the unit tangent vector to $\partial \Omega$. 
We remark that the above system of equations and their boundary conditions coincide with the limit of  (1.21) in \cite{Lewicka-Mahadevan-Pakzad_2011}, as $\lambda \to +\infty$, $\nu \to 1/2$, consistent with a characterization of incompressibility.


\bibliographystyle{siam} 
\bibliography{reference}


\end{document}